\documentclass[11pt]{amsart}
\usepackage{amsmath}
\usepackage{amssymb}
\usepackage{amsthm}
\usepackage{latexsym,tikz}
\usepackage{hyperref}
\usepackage{enumerate}

\newcommand{\diag}{{diag}}
\newcommand{\IM}{{Im}}
\newcommand{\RE}{{Re}}
\newcommand{\R}{\mathbb{R}}
\newcommand{\C}{\mathbb{C}}
\newcommand{\F}{\mathbb{F}}
\newcommand{\D}{\mathbb{D}}
\newcommand{\bH}{\mathbb{H}}
\newcommand{\Hq}{\mathbb{H}}
\newcommand{\bP}{\mathbb{P}}
\newcommand{\bS}{\mathbb{S}}
\newcommand{\M}{\mathcal{M}}
\newcommand{\conv}{{\rm conv}}
\newcommand{\bD}{\mathbb{D}}
 \newtheorem{thm}{Theorem}[section]
 \newtheorem{cor}[thm]{Corollary}
 \newtheorem{lem}[thm]{Lemma}
 
 \newtheorem{prop}[thm]{Proposition}

 \theoremstyle{definition}

 \theoremstyle{remark}
 \newtheorem{rem}[thm]{Remark}
 
  \newtheorem{example}[thm]{Example}
 \numberwithin{equation}{section}

\begin{document}

\thanks{The second author was partially supported by FCT
 through project UID/MAT/04459/2013 and the third author was partially supported by FCT through CMA-UBI, project PEst-OE/MAT/UI0212/2013.}

\title[A bridge between quaternionic and complex numerical ranges]{A bridge between quaternionic and complex numerical ranges}

\author[L. Carvalho]{Lu\'{\i}s Carvalho}
\address{Lu\'{\i}s Carvalho, ISCTE - Lisbon University Institute\\    Av. das For\c{c}as Armadas\\     1649-026, Lisbon\\   Portugal}
\email{luis.carvalho@iscte-iul.pt}
\author[Cristina Diogo]{Cristina Diogo}
\address{Cristina Diogo, ISCTE - Lisbon University Institute\\    Av. das For\c{c}as Armadas\\     1649-026, Lisbon\\   Portugal\\ and \\ Center for Mathematical Analysis, Geometry,
and Dynamical Systems\\ Mathematics Department,\\
Instituto Superior T\'ecnico, Universidade de Lisboa\\  Av. Rovisco Pais, 1049-001 Lisboa,  Portugal
}
\email{cristina.diogo@iscte-iul.pt}
\author[S. Mendes]{S\'{e}rgio Mendes}
\address{S\'{e}rgio Mendes, ISCTE - Lisbon University Institute\\    Av. das For\c{c}as Armadas\\     1649-026, Lisbon\\   Portugal\\ and Centro de Matem\'{a}tica e Aplica\c{c}\~{o}es \\ Universidade da Beira Interior \\ Rua Marqu\^{e}s d'\'{A}vila e Bolama \\ 6201-001, Covilh\~{a}}
\email{sergio.mendes@iscte-iul.pt}

%\date{\today}
\maketitle
%15B33 Matrices over special rings (quaternions, finite fields, etc.)

\begin{abstract}
We obtain a sufficient condition for the convexity of quaternionic numerical range for complex matrices in terms of its complex numerical range. It is also shown that the Bild coincides with complex numerical range for real matrices. From this result we derive that all real matrices have convex quaternionic numerical range. As an example we fully characterize the quaternionic numerical range of $2\times2$ real matrices.

\end{abstract}

%\tableofcontents

\maketitle
%%% ----------------------------------------------------------------------
%\tableofcontents
\section{Introduction}
The quaternionic numerical range \cite{Ki,R,STZ,Thompson,Ye1,Ye2,Zh}, the natural counterpart to the complex numerical range\footnote{From now on whenever a concept has no adjective it is because we are considering it over the quaternions.}\cite{GR}, is the image of a quadratic operator over the unit sphere. A striking difference between the quaternionic and complex numerical ranges is convexity. Contrary to what happens in the complex case, the quaternionic numerical range is not convex in general (see, for instance, \cite{R,Ye1,K}). Moreover, for a given finite operator it is very difficult to characterize its numerical range whose shape is, for the most part, unknown. A noticeable exception regarding the study of convexity and shape of the numerical range is the case of normal matrices (see \cite{Ye1,Ye2,STZ}).

The numerical range is a subset of a $4$-dimensional space and so it is hard to figure out its geometry. However we can still get a visualization of the numerical range through the Bild, its two dimensional equivalent, introduced by Kippenhahn \cite{Ki}.
%\footnote{For some author the Bild is the upper complex plane part of the numerical range, to refer to that we will use the term upper Bild}.
%Since the numerical range is a union of similarity classes, we only need to know one element in each similarity class to recover it completely. The Bild of a finite operator is the set of complex elements from all similarity classes and it is also an important tool to study the convexity of the numerical range. The equality of the Bild and the projection of the numerical range on the complex plane is a necessary and sufficient condition for convexity \cite{Ye1}.
Although tremendously useful to study the convexity and to represent the numerical range, to compute the Bild one needs first to compute the numerical range, which is algebraically very difficult. Even for $2\times2$ matrices the full characterization of the numerical range is still unknown.

In this paper we obtain a sufficient condition for complex matrices to have convex quaternionic numerical range. It turns out that this condition only involves properties of the complex numerical range of the same matrix (see theorem \ref{thm_convexity_qnr}). That is, to figure out if a quaternionic numerical range is convex we only need to compute the complex numerical range. We emphasize that the condition is not necessary, as a simple example shows, see remark \ref{remark convexity}.
For the class of real matrices, or their unitary equivalent, we conclude that the quaternionic numerical range is always convex (see theorem \ref{thm_reals}). Specifically, we establish that the Bild and the complex numerical range are equal. This result is not true in general as Thompson  has shown in \cite{Thompson}. Another consequence of this result is that we can now transport what happens in the complex case to the Bild and, by similarity, to the quaternionic numerical range (see theorem \ref{thm_eq_classes}). The strength of these results is that it breaks off the usual difficulties to calculate and understand what is the numerical range.
Using these results we fully characterize the numerical range of the $2\times2$ real matrices (see example \ref{ex_2x2_matrix}). In examples \ref{ex_matrix_p} and \ref{ex_block_matrix} we characterize the shape of certain classes of $3\times 3$ and $n\times n$ complex matrices.
 The number of examples we could add is as big as the number of results that we have for real matrices in the complex case. It should also be noted that results on some other features of the numerical range, as the numerical radius or the Crawford number, may also be fully transposed to the quaternionic case.

% matrix $B \in M_n(\C)$ and $\C^*$, that is we can not compute the upper bild of a specific matrix $A$, as the upper part of the complex numerical range of any matrix $B$. What the proof shows is that for any real matrix $A$ the (upper) bild is equal to the (upper) complex numerical range of the matrix $A$.
%
%\[
%W_{Hq}(A) \cap \C^+ \neq W_{\C}(B) \cap \C^+, \;\;\;\; \text{ for any } B \in M_k(\C)
%\]

\section{Preliminaries}

In this section we present some well known facts about quaternions and fix some notation. The quaternionic skew-field $\bH$ is an algebra of rank $4$ over $\R$ with basis $\{1, i, j, k\}$. The product in $\bH$ is given by $i^2=j^2=k^2=ijk=-1$. Denote the pure quaternions by $\bP=\mathrm{span}_{\R}\,\{i,j,k\}$. For any $q=a_0+a_1i+a_2j+a_3k\in\bH$ let $Re(q)=a_0$ and $Im(q)=a_1i+a_2j+a_3k$ be the real and imaginary parts of $q$, respectively. The conjugate of $q$ is given by $q^*=Re(q)-Im(q)$ and the norm is defined by $|q|^2=qq^*$. Two quaternions $q_1,q_2\in\bH$ are called similar if there exists a unitary quaternion $s$ such that $s^{*}q_2 s=q_1$. Similarity is an equivalence relation and we denote by $[q]$ the equivalence class containing $q$. A necessary and sufficient condition for the similarity of $q_1$ and $q_2$ is that $Re(q_1)=Re(q_2) \textrm{ and }|Im(q_1)|=|Im(q_2)|$.

Let $\F$ denote $\R$, $\C$ or $\bH$. Let $\F^n$ be the $n$-dimensional $\F$-space. The disk with centre $a\in\F^n$ and radius $r>0$ is the set $\D_{\F^n}(a,r)=\{x\in\F^n:|x-a|\leq r\}$ and its boundary is the sphere $\bS_{\F^n}(a,r)$. In particular, if $a=0$ and $r=1$, we simply write $\D_{\F^n}$ and $\bS_{\F^n}$. With this notation, the group of unitary quaternions is denoted by $\bS_{\bH}$.

Let $\M_{n} (\F)$ be the set of all $n\times n$ matrices with entries over $\F$. For $A\in\M_{n} (\F)$ ,  $\bar{A}$ and $A^*$ denote the conjugate and the conjugate transpose of $A$, respectively.
The set \[W_{\F}(A)=\{x^*Ax:x\in \bS_{\F^n}\}\]
is called the numerical range of $A$ in $\F$. It is well known that the numerical range is invariant under unitary equivalence, i.e., $W_{\F}(A)=W_{\F}(U^*AU)$, for every unitary $U\in \M_n(\F)$.

Let $\F=\bH$. It is well known that if $q\in W_{\bH}(A)$ then $[q]\subseteq W_{\bH}(A)$. Therefore, it is enough to study the subset of complex elements in each similarity class. This set is known as $B(A)$, the Bild of $A$
\[
B(A)=W_{\bH}(A)\cap\C
\]
Although the Bild may not be convex, the upper bild $B^+(A)=W_{\bH}(A)\cap\C^+$ is always convex, see \cite{Zh}.

Taking into account that $\F$ can be seen as a real subspace of $\Hq$, we denote the projection of $\Hq$ over $\F$ by $\pi_{\F}:\Hq \rightarrow \F$. The projection of $W_{\Hq}(A)$ over $\F$ is
\begin{equation*}
  \pi_{\F}(W_{\Hq}(A))= \{\pi_{\F}(w): \; w\in W_{\Hq}(A)\}.
\end{equation*}

%\[ \pi_\R\big(W_{\Hq}(A)\big)=\{Re(q): \; q\in W_{\Hq}(A)\}
%\]
%and
%\[\pi_\C\big(W_{\Hq}(A)\big)=\{co(q): \; q\in W_{\Hq}(A)\},
%\]
%however some authors use the notation $W_\Hq(A:\R)$ and $W_\Hq(A:\C)$, see for instance

Given $A \in \M_n(\Hq)$ there exists an associated complex matrix
\[
\chi(A)=\left[
\begin{array}{cc}
A_1 &  A_2 \\
%\hline
-\bar A_2 & \bar A_1
\end{array}
\right]\in\M_{2n}(\C),
\]
where $A_1, A_2 \in \M_n(\C)$ and $A=A_1+A_2j$.

Au-Yeung found \cite{Ye1} necessary and sufficient conditions for the convexity of $W_{\Hq}(A)$. One of these conditions is an equality between the Bild and the projection over $\C$ of the numerical range. In \cite{Ye2}, he proved that the projection over $\C$ of the numerical range is the complex numerical range of $\chi(A)$.

\begin{thm}\label{theo_convex_yeung}\cite{Ye1, Ye2}
  Let $A\in \M_n(\Hq)$. Then $W_{\bH}(A)$ is convex if and only if one the following statements hold:
  \begin{enumerate}[(i)]
    \item $ W_{\bH}(A)\cap\C=\pi_{\C}(W_{\bH}(A))=W_{\C}(\chi_A)$;
    \item $W_{\bH}(A)\cap\R=\pi_{\R}(W_{\bH}(A))$.
  \end{enumerate}
\end{thm}

\section{On the convexity of the numerical range}

%\section{Complex matrices}

Firstly, it should be noted that the complex numerical range is invariant under the transpose operator. This is a trivial conclusion of
$x^*A x=\big(x^*A x\big)^t=x^tA^t\overline{x}$, for $x\in \bS_{\C^n}$.
We have:
\begin{lem}\label{transpose matrices}
Let $A \in\M_n(\C)$. Then $W_{\C}(A)=W_{\C}(A^t).$
\end{lem}

The next proposition is the stepping stone of further results in the paper. It provides a more intuitive formulation of $W_{\C}(\chi_A)$ when $A$ is a complex matrix. As usual, $\conv(E)$ denotes the convex hull of $E\subseteq \C$.
%The next result is the stepping stone of further results in the paper. There is a well known necessary a sufficient condition, see \cite[p. 280]{Ye2} and \cite[Theorem 2]{Ye1}, that relates the complex numerical range of $\chi_A$ and the convexity of the numerical range of $A$. It states that $W_{\bH}(A)$ is convex if, and only if, $W_{\bH}(A)\cap\C=W_{\C}(\chi_A)$. This result provides a more intuitive formulation of $W_{\C}(\chi_A)$ when $A$ is a complex matrix.
%
\begin{prop}\label{complex matrices}
Let $A \in\M_n(\C)$. Then $W_{\C}(\chi_A)=\conv\{ W_{\C}(A), W_{\C}( A^*)\}$.
\end{prop}
\begin{proof}
Since $A$ is complex,
\[
\chi_A=\left(
                  \begin{array}{cc}
                    A & 0 \\
                    0 & \bar A \\
                  \end{array}
                \right).
\]
%For $\vc{w} \in \F^{n}$, $\vc{w} =\beta_w \vc{z}_w$  with $\vc{z}_w \in \bS_{\F^{n}}$.
Then, denoting $x=\beta_xz_x$, $y=\beta_yz_y$ with $z_x,z_y\in\bS_{\C^{n}}$ and $\beta_x=|x|, \beta_y=|y|$, we have
\begin{align*}
W_{\C}(\chi_A)=&\Big\{(x^* y^*) \chi_A \left(\begin{array}{c} x \\ y \\ \end{array}\right): (x \;  y) \in \bS_{\C^{2n}}\Big\}\\
              =&  \Big\{x^*A x+y^*\bar A y: (x \;  y) \in \bS_{\C^{2n}}\Big\} \\
              =&  \Big\{ \beta_x^2 z_x^* A z_x  + \beta_y^2 z_y^* \bar A z_y :  z_x, z_y  \in \bS_{\C^{n}}, (\beta_x, \beta_y) \in \D_{\R^2} \Big\}\\
              =& \Big\{ \alpha z_x^*A z_x+(1-\alpha) z_y^*\bar A z_y: z_x, z_y  \in \bS_{\C^{n}}, \alpha \in [0,1] \Big\}\\
              =& \Big\{ \alpha \omega_x+(1-\alpha) \omega_y:\omega_x\in W_{\C}(A),  \omega_y \in W_{\C}(\bar A), \alpha \in [0,1] \Big\}\\
              =&\,\conv\{ W_{\C}(A), W_{\C}( A^*)\}
\end{align*}
where the last equality is a consequence of convexity of the complex numerical range, lemma \ref{transpose matrices} and the following equality
\[
\Big\{ \alpha \omega_x+(1-\alpha) \omega_y:\omega_x\in W_{\C}(A),  \omega_y \in W_{\C}(\bar A), \alpha \in [0,1] \Big\} \]
\[= \Big\{ \alpha \omega_x+(1-\alpha) \omega_y:\omega_x,\omega_y\in W_{\C}(A)\cup \overline{W_{\C}(A)}, \alpha \in [0,1] \Big\}.
\]
\end{proof}
By theorem \ref{theo_convex_yeung}, we have that $W_{\bH}(A)$ is convex if, and only if, $W_{\bH}(A)\cap\C=W_{\C}(\chi_A)$. From Proposition \ref{complex matrices} it follows:
\begin{cor} \label{corol_complex_conv}
Let $A \in\M_n(\C)$. $W_{\bH}(A)$ is convex if, and only if, $W_{\bH}(A) \cap \C=\conv\{W_{\C}(A), W_{\C}(A^*)\}$.
\end{cor}
Next theorem gives a sufficient condition for the convexity of quaternionic numerical range of a complex matrix in terms of the complex numerical range. This condition is a complex analogue of the well known result which states that $W_{\bH}(A)$ is convex if and only if $\pi_{\R}(W_{\bH}(A))=W_{\bH}(A) \cap \R$ (see (ii) in theorem \ref{theo_convex_yeung}).

\begin{thm}\label{thm_convexity_qnr}
Let $A \in\M_n(\C)$. If $\pi_{\R}(W_{\C}(A))=W_{\C}(A) \cap \R$ then  $W_{\bH}(A)$ is convex.
\end{thm}
\begin{proof}
We begin by proving the following result:

\bigskip

$(\ast)\,\,$If $B\subset\C$ is convex then $B\cup\overline{B}$ is convex if, and only if, $\pi_{\R}(B)=B\cap\R$.

\bigskip

Suppose $B\cup\overline{B}$ is convex. Then, given $\lambda=a+ib\in B$,
$$\frac{1}{2}\lambda+\frac{1}{2}\overline{\lambda}=a=\pi_{\R}(\lambda)$$
and so $\pi_{\R}(\lambda)\in B\cap\R$.

For the converse, suppose $B\cup\overline{B}$ is non convex and consider $\omega=\alpha\lambda_1+(1-\alpha)\overline{\lambda_2}\notin B\cup\overline{B}$, with $0\leq\alpha\leq 1$ and $\lambda_1, \lambda_2\in B$. We claim that $\pi_{\R}(\omega)\notin B$. In fact, if $\pi_{\R}(\omega)\in B$, there is a point $\widetilde{\omega}$ in the segment $[\lambda_1,\lambda_2]\subset B$ with $\pi_{\R}(\omega)=\pi_{\R}(\widetilde{\omega})$ such that $\omega=\beta\pi_{\R}(\omega)+(1-\beta)\widetilde{\omega}$, for some $\beta\in]0,1[$, which is impossible since $B$ is convex. Since
$$\pi_{\R}(\omega)=\alpha \pi_{\R}(\lambda_1)+(1-\alpha)\pi_{\R}(\lambda_2)$$
we conclude that either $\pi_{\R}(\lambda_1)\notin B$ or $\pi_{\R}(\lambda_2)\notin B$ for if $\pi_{\R}(\lambda_1), \pi_{\R}(\lambda_2)\in B$ we would have $\pi_{\R}(\omega)\in B$.

Now, suppose $\pi_{\R}(W_{\C}(A))=W_{\C}(A)\cap\R$. From $(\ast)$, since $W_{\C}(A)$ is convex, $W_{\C}(A)\cup W_{\C}(A^*)=W_{\C}(A)\cup\overline{W_{\C}(A)}$ is convex and so
$$\conv\{W_{\C}(A),W_{\C}(A^*)\}=W_{\C}(A)\cup W_{\C}(A^*).$$
On the other hand, by corollary \ref{corol_complex_conv}, $W_{\bH}(A)$ is convex if, and only if, $W_{\bH}(A)\cap\C=W_{\C}(A)\cup W_{\C}(A^*)$. Since $\C\subset\bH$, we may identify $W_{\C}(A)$ with a subset of $W_{\bH}(A)$. On the other hand, $W_{\C}(A^*)= \overline{W_{\C}(A)}$ and, by similarity, $W_{\C}(A^*)\subset W_{\bH}(A)$. Hence,
$W_{\C}(A)\cup W_{\C}(A^*)\subseteq W_{\bH}(A)\cap\C$.

The converse inclusion comes from
\begin{align*}
 W_{\C}(A)\cup W_{\C}(A^*) & = \conv\{W_{\C}(A),W_{\C}(A^*)\}\\
   & = W_{\C}(\chi_A) \\
   & = \pi_\C(W_{\bH}(A)) \\
   & \supseteq W_{\bH}(A)\cap\C.
\end{align*}
\end{proof}
\begin{rem}\label{remark convexity}
The previous sufficient condition is not necessary as a simple example clarifies. Take $A=\diag(i, 2i) \in \M_2(\C)$. We claim that the numerical range is the disk over the pure quaternions of radius $2$ centered at zero, $W_{\Hq}(A)=\D_{\bP}(0,2)$.

Let $w=x^*Ax\in W_{\bH}(A)$, where $x\in (x_1, x_2)\in \bS_{\bH^2}$. We may write
\begin{equation}\label{omega}
w=\beta_1^2z_1^*iz_1+2\beta_2^2z_2^*iz_2
\end{equation}
where $z_1,z_2\in\bS_{\bH}$ and $\beta_1=|x_1|,\beta_2=|x_2|$. Since $Re(ab)=Re(ba)$, clearly, $Re(w)=0$ and, by the triangle inequality,
\[
|w|\leq \beta_1^2+2\beta_2^2=\beta_1^2+2(1-\beta_1^2)=2-\beta_1^2\leq 2.
\]
Hence, $w\in\bD_{\bP}(0,2)$. Conversely, by similarity, it is enough to show that
\[
\bD_{\bP}(0,2)\cap\C^+\subseteq W_{\bH}(A)\cap\C^+.
\]
The set $\bD_{\bP}(0,2)\cap\C^+$ is the segment $[0,2i]$ and, since the upper Bild is convex \cite{Zh}, we only need to prove that $0, 2i\in W_{\bH}(A)\cap\C^+$. Taking $\beta_1=0, z_1=1, \beta_2=1$ and $z_2=1$ in (\ref{omega}) we have that $2i\in W_{\bH}(A)\cap\C^+$. To show that $0\in W_{\bH}(A)\cap\C^+$ simply take $\beta_1=\sqrt{\frac{2}{3}}, z_1=1, \beta_2=\sqrt{\frac{1}{3}}$ and $z_2=j$ in (\ref{omega}).

From the previous discussion we conclude that $W_{\Hq}(A)$ is convex. On the other hand, the complex numerical range is the segment joining $i$ and $2i$, $W_{\C}(A)=[i, 2i]$, and
\[
\pi_{\R}(W_{\C}(A))=\{0\}\neq \emptyset= W_{\C}(A) \cap \R.
\]
\end{rem}

When the matrix is real we can improve further our previous results since in this case $W_{\C}(A)=W_{\C}(A^*)$. From theorem \ref{theo_convex_yeung} and proposition \ref{complex matrices} we have
\[
\pi_\C(W_{\bH}(A)) = W_{\C}(\chi_A)= W_{\C}(A).
\]
Therefore,
\begin{equation}\label{W_H(A) cap C=W_C(A}
W_{\C}(A)\subseteq W_{\bH}(A)\cap\C\subseteq \pi_\C(W_{\bH}(A))=W_{\C}(A).
\end{equation}

From the above we see that the Bild coincides with the complex numerical range. Using theorem \ref{theo_convex_yeung}, we conclude the following:

%From \cite[Theorem 2]{Ye1}, $W_{\bH}(A)$ is convex if, and only if, $W_{\bH}(A)\cap\C=\pi_\C(W_{\bH}(A))$. We have:
%\begin{align*}
% \pi_\C(W_{\bH}(A)) & = W_{\C}(\chi_A) \\
 %   & = W_{\C}(A)\,,
%\end{align*}
%where the first equality follows from \cite[p. 280]{Ye2} and the second from proposition \ref{complex matrices}. But then
%\begin{equation}\label{W_H(A) cap C=W_C(A}
%W_{\C}(A)\subseteq W_{\bH}(A)\cap\C\subseteq \pi_\C(W_{\bH}(A))=W_{\C}(A).
%\end{equation}

\begin{thm}\label{thm_reals}
If $A \in\M_n(\R)$ then $W_{\bH}(A)$ is convex.
\end{thm}

This result is not true in general since Thompson \cite{Thompson} proved the existence of a quaternionic matrix whose upper Bild does not coincide with the upper part of the complex numerical range for any complex matrix. We can even find a way to compute $W_{\bH}(A)$ out of the complex numerical range of $A$. In fact, since $W_{\bH}(A)$ is given by the equivalence classes of the Bild $W_{\Hq}(A)\cap\C$, from (\ref{W_H(A) cap C=W_C(A}) we conclude:

\begin{thm}\label{thm_eq_classes}
Let $A \in\M_n(\R)$. Then, $W_{\Hq}(A)=\Big[W_{\C}(A)\Big]$.
\end{thm}

The above result essentially says that the quaternionic numerical range of a real matrix $A$ corresponds to the rotation in $\bH$, over the reals, of the complex numerical range of $A$.

\section{Examples and applications}

There is a vast class of complex matrices to which we can apply the previous results, namely theorem \ref{thm_convexity_qnr} and theorem \ref{thm_eq_classes}, to conclude about convexity  and shape of quaternionic numerical range. In this section we give some examples that show how we can transport some results from complex numerical range to the quaternionic setting. However, we would like to stress that there are as many examples as the known results for complex numerical range.

The first example provides a full characterization of the numerical range of real $2\times2$ matrices. Naturally, as in the complex case, we have three possible and distinct cases: the numerical range is a line segment, a disk or an ellipsoid.

\begin{example}\label{ex_2x2_matrix}

Let $A \in M_2(\R)$. By the Elliptical Range Theorem  \cite{GR}, the complex numerical range of $A$ is an elliptical disc with foci $\lambda_1, \lambda_2$, at the eigenvalues of $A$. Notice that the numerical range of $A$ can be a line segment or a disk, which can be viewed as a degenerated ellipse. Since the matrix $A$ is real, its eigenvalues are real or a pair of complex conjugate.

If $A$ is normal, then $A$ is unitarily similar to a diagonal matrix $D=\diag(\lambda_1, \lambda_2)$. Then $W_{\Hq}(A)=W_{\C}(A)=[\lambda_1, \lambda_2]$, if $\lambda_1, \lambda_2\in\R$ or $W_{\Hq}(A)=\Big[W_{\C}(A)\Big]= \Big[[\lambda_1,\overline{\lambda_1}]\Big] =\D_{\bP} (\RE (\lambda_1), |\IM (\lambda_1)|)$, if $\lambda_2=\overline{\lambda_1}$, by theorem \ref{thm_eq_classes}.

If $A$ is not normal, then $A$ is unitarily equivalent to an upper triangular matrix
\[B=\left(
                 \begin{array}{cc}
                   \lambda_1 & \omega \\
                   0 & \lambda_2 \\
                \end{array}
              \right), \quad \text{with}\quad \omega\neq 0.\]
If $\lambda_2=\lambda_1$, then $W_{\Hq}(B)=\Big[W_{\C}(B)\Big]=\Big[\D_{\C}\left(\lambda_1, \frac{|\omega|}{2}\right)\Big] =\D_{\Hq}\left(\lambda_1, \frac{|\omega|}{2}\right)$.
If $\lambda_2\neq \lambda_1$, $W_{\C}(B)$ is the elliptical disk
\[W_{\C}(B)=\Bigg\{x+iy\in\C: \frac{(x-c)^2}{a^2}+\frac{y^2}{b^2}\leq 1\Bigg\},\]
where $c=\frac{\lambda_1+\lambda_2}{2}$ and $a, b$ are the semiaxes.
Let $q=z_1+z_2i+z_3j+z_4k\in\Hq$. From theorem \ref{thm_eq_classes} we have  $W_{\Hq}(B)=\Big[W_{\C}(B)\Big]$, so
\begin{eqnarray*}
  W_{\Hq}(B) &=& \{[x+iy]: x+iy\in W_{\C}(B) \} \\
   &=& \{q\in\Hq: \RE(q)=x,  \; |\IM(q)|^2=y^2, \text{for} \; x+iy\in W_{\C}(B) \} \\
   &=&  \Big\{q\in\Hq: \frac{(\RE(q)-c)^2}{a^2}+\frac{|\IM(q)|^2}{b^2}\leq 1\Big\}\\
   &=&  \Big\{q\in\Hq: \frac{(z_1-c)^2}{a^2}+\frac{z_2^2}{b^2}+\frac{z_3^2}{b^2}+\frac{z_4^2}{b^2}\leq 1\Big\}
\end{eqnarray*}

Therefore, $W_{\Hq}(A)$ is a line segment, a disk or a 4-dimensional ellipsoid.
\end{example}

%%%%%%%%%%%%%%%%%%%%%%%%%%%%%%%%%%%%%%%%%%%%%%%%%%%%%%%%%%%%%%%%%%%%%%%%%%%%%%%%%%%%%%%%%%%%%%

\begin{example}\label{ex_matrix_p}
Let $A\in M_3(\C)$ be the matrix
\begin{equation}\label{matrixp}
A=\left(
                 \begin{array}{ccc}
                  p & x & y \\
                   0 & p & z \\
                    0 & 0 & p\\
                \end{array}
              \right),
\end{equation}
with $p\in\R$ and $xyz=0$. From \cite[Theorem 4.1]{KRS}, $W_\C(A)$ is a disk with centre $p$ and radius $r=\frac{1}{2}\sqrt{|x|^2+|y|^2+|z|^2}$. In this case $\pi_{\R}(W_{\C}(A))=W_{\C}(A) \cap \R$ so, from theorem \ref{thm_convexity_qnr}, we have that $W_{\Hq}(A)$ is convex.

When, in addition, we have $x,y,z\in \R$, i.e., $A$ is a real matrix with $xyz=0$, we can characterize the shape of the quaternionic numerical range of $A$. Since $\Big[\D_{\C}(p,r)\Big]=\D_{\Hq}(p,r)$, it follows from theorem \ref{thm_eq_classes} that \[W_{\Hq}(A)=\D_{\Hq}(p,r).\]
\end{example}

%%%%%%%%%%%%%%%%%%%%%%%%%%%%%%%%%%%%%%%%%%%%%%%%%%%%%%%%%%%%%%%%%%%%%

\begin{example}\label{ex_block_matrix}
According to  \cite[Corollary 2.3]{BS}, if $A\in M_n(\C)$ is unitarily equivalent to a matrix of the form
\begin{equation}\label{matrixp}
A=\left(
                 \begin{array}{cc}
                   a_1I_{n_1} & X \\
                   k X^* & a_2I_{n_2} \\
                \end{array}
              \right),
\end{equation}
then $W_{\C}(A)$ is an ellipse (see formulas of foci in \cite[Corollary 2.3]{BS}). If the foci of the ellipse are real or a pair of complex conjugate, the centre of the ellipse is real and
 $\pi_{\R}(W_{\C}(A))=W_{\C}(A) \cap \R$. By theorem \ref{thm_convexity_qnr} we have $W_{\Hq}(A)$ is convex.

We can say more about the shape of the numerical range when $A$ is a real matrix. From theorem \ref{thm_eq_classes} and using the same reasoning of example  \ref{ex_2x2_matrix} it follows that the quaternionic numerical range of A is a 4-dimensional ellipsoid centered at the real line.
\end{example}

%\begin{example}\label{ex_tridiagonal}
%Let $A\in M_n(\C)$ be a tridiagonal matrix, i.e., $a_{ij}=0$ whenever $|i-j|>1$. Then, $A$ has the form
%\begin{equation}\label{tridiagonal}
%A=\left(
        %         \begin{array}{ccccc}
         %          a_1  & b_1 & 0 & \dots & 0 \\
        %           c_1 & a_2 & b_2 &\ddots & \vdots\\
       %            0  & c_2 & a_3 & \ddots & 0\\
  %                  \vdots & \ddots  & \ddots & \ddots & b_{n-1}\\
  %                 0 & \dots & 0 & c_{n-1} & a_n\\
    %            \end{array}
   %           \right).
%\end{equation}
%Suppose $a_j=\left\{
  %             \begin{array}{ll}
 %                a_1, &  \hbox{if $j$ odd;} \\
 %                a_2, & \hbox{if $j$ even,}
 %              \end{array}
 %            \right.
%$
%and, for each $j=1,\dots, n-1$, either $k\overline{b_j}=c_j$ or $k\overline{c_j}=b_j$, for $k\in \C$ constant. From \cite[Theorem 3.3]{BS} it follows that the $W_{\C}(A)$ is also an ellipse. Again, if the foci are complex conjugate then, from Theorem \ref{thm_convexity_qnr}, $W_{\Hq}(A)$ is convex.

%\end{example}

%\begin{rem}
%A curious implication of the previous theorem is the case of the matrix with real eigenvalues and a quaternionic entry above the diagonal (not necessarily complex). This matrix is unitarily equivalent to a real matrix with the same real entries in the diagonal. Therefore its numerical range is also an ellipsoid.
%\end{rem}

%d we have:
%\[
%W_{\C}(\chi_A)=\Big\{y^*Ay: y \in \bS_{\C^{n}}\Big\}=W_{\C}(A)
%\]
%The case $y=0$ is similar.

% ------------------------------------------------------------------------

\begin{thebibliography}{99}
\bibitem[BS]{BS} E. Brown, I. Spitkovsky, \emph{On matrices with elliptical numerical ranges}. Linear and
Multilinear Algebra, \textbf{52:3-4} (2004), 177--193, DOI:10.1080/0308108031000112589

\bibitem[GR]{GR} K. Gustafson, D. Rao, \emph{Numerical Range}, Springer-Verlag, New York, 1997.

\bibitem[Ki]{Ki} R. Kippenhahn, O\emph{n the numerical range of a matrix}, Translated from the German by Paul F. Zachlin and Michiel E. Hochstenbach. Linear Multilinear Algebra \textbf{56:1-2 }(2008), 185-225.

\bibitem[KRS]{KRS} D. Keeler, L. Rodman, I. Spitkovsky, \emph{The numerical range of $3\times 3$ matrices}, Linear Algebra and
its Applications, \textbf{252} (1997), 115--139

\bibitem[K]{K} P. Kumar, \emph{A note on convexity of sections of quaternionic numerical range}, Linear Algebra and its Applications, \textbf{572} (2019), 92--116.

\bibitem[R]{R} L. Rodman, \emph{Topics in Quaternion Linear Algebra}, Princeton University Press, 2014.

\bibitem[STZ]{STZ} W. So, R. Thompson, F. Zhang, \emph{The numerical range of normal matrices with quaternion entries}, Linear and Multilinear Algebra, \textbf{37} (1994), 175--195.

\bibitem[To]{Thompson}  R. Thompson, \emph{The upper numerical range of a quaternionic matrix is not a complex numerical range}, Linear Algebra and its Applications, \textbf{254:1-3} (1997), 19-28.

\bibitem[Ye1]{Ye1} Y. Au-Yeung ,\emph{\emph{ On the convexity of the numerical range in quaternionic Hilbert space}}, Linear and
Multilinear Algebra, \textbf{16} (1984), 93--100

\bibitem[Ye2]{Ye2} Y. Au-Yeung , \emph{A short proof of a theorem on the numerical range of a normal
quaternionic matrix}, Linear and Multilinear Algebra, \textbf{39:3} (1995), 279--284, DOI: 10.1080/03081089508818402

\bibitem[Zh]{Zh} F. Zhang, \emph{Quaternions and matrices of quaternions}, Linear Algebra and its Applications, \textbf{251} (1997), 21--57

\end{thebibliography}
\end{document}